\documentclass[12pt, psamsfonts]{amsart}
\usepackage{amsmath}
\usepackage{amsthm}
\usepackage{amssymb}
\usepackage{amscd}
\usepackage{amsfonts}
\usepackage{amsbsy}
\usepackage{epsfig}
\usepackage{ulem}
\usepackage[latin1]{inputenc}      % pour les accents
\usepackage{tikz}
\newtheorem{theorem}{Theorem}
\newtheorem{lemma}[theorem]{Lemma}
\newtheorem{proposition}[theorem]{Proposition}
\newtheorem{corollary}[theorem]{Corollary}

\newtheorem{remark}[theorem]{Remark}

\newtheorem{definition}[theorem]{Definition}
\newtheorem{example}[theorem]{Example}
\newtheorem*{claim*}{Claim}

\newfont\bbf{msbm10 at 12pt}

\def\eps{\varepsilon}

\def\N{{\mathbb N}}

\def\A{{\mathcal A}}

\def\J{{\mathcal J}}

\def\G{{\mathcal G}}

\def\B{{\mathcal B}}

\def\Q{{\mathbb Q}}

\def\A{{\mathcal A}}
\def\B{{\mathcal B}}

\def\Fix{{\hbox{{\rm Fix}}}}
\def\d{{\hbox{{\rm d}}}}

\def\id{{\hbox{{\rm id}}}}

\def\supp{\mbox{\rm supp}}

\def\card{\mbox{\rm card} }

\def\orb{\mbox{\rm orb}}

\def\le{\leqslant}
\def\ge{\geqslant}

\def\1{ {\hbox{{\it 1}} \!\! I} }

\begin{document}

\title[Typical properties of interval maps]
{Typical properties of interval maps  \\ preserving the Lebesgue measure}
\author{Jozef Bobok}

\author{Serge Troubetzkoy}

\address{Department of Mathematics of FCE\\Czech Technical University in Prague\\
Th\'akurova 7, 166 29 Prague 6, Czech Republic}
\email{jozef.bobok@cvut.cz}

\address{Aix Marseille Univ, CNRS, Centrale Marseille, I2M, Marseille, France\linebreak
postal address: I2M, Luminy, Case 907, F-13288 Marseille Cedex 9, France}
\email{serge.troubetzkoy@univ-amu.fr}
\urladdr{www.i2m.univ-amu.fr/perso/serge.troubetzkoy/} \date{}

\thanks{We thank the A*MIDEX project (ANR-11-IDEX-0001-02), funded itself by the ``Investissements d'avenir'' program of the French Government, managed by the
French National Research Agency
(ANR).
 The first  author was supported by the European Regional Development Fund, project No.~CZ 02.1.01/0.0/0.0/16\_019/0000778.
 We thank two anonymous referees as well as Pierre-Antoine  Guiheneuf for suggestions which lead to numerous improvement to our exposition.}

\subjclass[2000]{37E05, 37B40, 46B25, 46.3}
\keywords{weakly mixing, leo, entropy}

\begin{abstract} We consider the class of the continuous functions from $[0,1]$ into itself which preserve the Lebesgue measure. This class endowed
with  the uniform metric constitutes a complete metric space. We investigate the dynamical properties of typical maps from the space.
\end{abstract}

\maketitle
\section{Introduction and summary of  results}
This article is about typical properties of continuous maps of the interval which preserve the Lebesgue measure.  Throughout the article the word
typical will  mean that  with respect to the uniform topology there is a dense $G_\delta$ set of maps having this property.  Such results
are in the domain of approximation theory. To our knowledge,
the use of approximation techniques in dynamical systems was
started in 1941 by Oxtoby and Ulam  who considered a
simplicial polyhedron with a non-atomic measure which
is positive on open sets. In this setting they showed that the set of ergodic measure-preserving homeomorphisms
is typical in the strong topology  \cite{OxUl41}.
In 1944 Halmos  introduced approximation techniques in a
purely metric situation: the study of invertible mode 0 maps of the interval $[0,1]$ which preserve the Lebesgue measure. He showed
 that the typical invertible map is weakly mixing, i.e., has continuous spectrum \cite{Ha44},\cite{Ha44.1},\cite{Ha56}.
In 1948  Rohlin  showed that the set of (strongly) mixing measure preserving invertible maps is of the  first category \cite{Ro48}.
 In  1967 Katok and Stepin \cite{KaSt67}  introduced the notation of speed of approximation. One of the notable applications of their method is
 the typicality of ergodicity and weak mixing for certain
classes of interval exchange transformations. Katok has shown that interval exchange transformations are never mixing \cite{Ka80}.

The study of typical properties of homeomorphism of compact manifolds  which preserve regular measures was continued and generalized
by Katok and Stepin in 1970 \cite{KaSt70}
who showed the typicality of ergodicity, of simple continuous spectrum, and the absence of mixing.
In 1980 Yano showed that the generic homeomorphism of a compact manifold has infinite entropy \cite{Ya}. 
More recently Alpern and co-authors have unified the studies of homeomorphisms and of measure preserving transformations and shown that
if a property is typical for measure preserving transformations then it is typical for homeomorphisms \cite{Al78}, \cite{Al79}, \cite{AlPr00}.
Recently many authors have shown that that the
shadowing property is generic in various setting: \cite{Br}, \cite{GuLe18}, \cite{KoMaOpKu}, \cite{Ko},\cite{Pi}.

Except for Yano's result which also holds for generic continuous maps of compact manifolds, all  of these results are about invertible maps. The measure
theoretic properties of $C^0$-generic systems on compact manifolds have
first been studied by Abdenur and Andersson \cite{AA}; their main result is that a $C^0$-generic map has no physical measure,
however the Birkhoff average of any continuous function is convergent for Lebesgue a.e.\ point.
Catsigeras and Troubetzkoy  gave a quite detailed description of the invariant measures
of $C^0$ generic continuous maps of compact manifolds with or without boundary (as well as for  generic homeomorphisms in the same setting) \cite{CT1},
\cite{CT2}, \cite{CT3}.

Many more more details of the history of approximation theory can be found in the surveys \cite{ChPr}, \cite{BeKwMe}.

 In this article we consider the set $C(\lambda)$ if continuous non-invertible maps of the unit interval $[0,1]$
which preserve the Lebsegue
measure $\lambda$.  Every such map has a dense set of periodic points.  Furthermore, except for the two exceptional maps $id$ and $1-id$, every
such map has positive metric entropy.
The $C(\lambda)$-typical function  (all the properties will be defined later in the article):
\begin{itemize}
\item[(i)] is weakly mixing with respect to $\lambda$ (Theorem \ref{t:5}),
\item[(ii)] is leo (Theorem \ref{t:4}),
\item[(iii)] satisfies the periodic specification property (Corollary \ref{blokh})
\item[(iv)] has a knot point at $\lambda$ almost every point \cite{Bo91},
\item[(v)] maps a set of Lebesgue measure zero onto $[0,1]$ (Corollary \ref{cor:onto}),
\item[(vi)] has infinite topological entropy (Proposition \ref{prop:inf}),
\item[(vii)] has Hausdorff dimension = lower Box dimension = 1 $<$ upper Box dimension  = 2 \cite{ScWi95}.
\end{itemize}
Furthermore, in analogy to Rohlin's result, we show that the set of mixing maps in $C(\lambda)$ is dense (Corollary \ref{p:4}) and of the first category (Theorem
\ref{t:8}). We
also show that for
any $c > 0$ as well as for $c =
\infty$
the set of
maps having metric entropy $c$ is dense in $C(\lambda)$,
however, we do not know if there is a value $c$ such that this set is generic.

Points  i) and vi) and the furthermore results are analogous to results in the invertible case, and the proofs of these results
follow the same general plan as in the invertible case. The other points have no analogies in the invertible case

Let $\mu$ be a probability Borel measure with full support, let  $C(\mu)$ be the set of all continuous interval maps
preserving the measure $\mu$ equipped with the
uniform metric. The map $h\colon~[0,1]\to [0,1]$ defined as $h(x)=\mu([0,x])$ for  $x\in [0,1]$, is a homeomorphism of $[0,1]$.  In fact, in this settings
$\lambda=h^*\mu$, i.e., $\lambda$ is the
pushforward measure of  $\mu$ by $h$. Moreover, the map $H\colon~C(\lambda)\to C(\mu)$ given by $f\longmapsto h^{-1}\circ f \circ h$ is a homermoprhism of the
spaces $C(\lambda)$ and
$C(\mu)$. Using $H$, we can transfer the properties (i)-(vi) listed above into the context of $C(\mu)$. In particular, the property (iv) in $C(\mu)$ says that
$C(\mu)$-typical function has a
knot point at $\mu$ almost every point \cite{Bo91}.

\section{Maps in $C(\lambda)$}
Let  $\lambda$ denote the Lebesgue measure on $[0,1]$ and $\B$ the Borel sets in $[0,1]$. Let $C(\lambda)$ consist of all continuous $\lambda$-preserving
functions from $[0,1]$ onto $[0,1]$,
i.e.,
$$ C(\lambda)=\{f\colon~[0,1]\to [0,1]\colon~\forall
A\in\B,~\lambda(A)=\lambda(f^{-1}(A))\}. $$
We consider the uniform metric $\rho$ on $C(\lambda)$:
$\rho (f,g) := \sup_{x \in [0,1]} |f(x) - g(x)|$.

\begin{proposition}\label{p:1}$(C(\lambda),\rho)$  is a complete metric space. \end{proposition}
We leave the standard proof of this result to the reader.
\begin{definition}We say that continuous maps $f,g\colon~[a,b]\subset [0,1]\to [0,1]$ are $\lambda$-equivalent if for each Borel set $A\in\B$,
$$\lambda(f^{-1}(A))=\lambda(g^{-1}(A)).
$$
For $f\in C(\lambda)$ and $[a,b]\subset [0,1]$ we denote by $C(f;[a,b])$ the set of all continuous maps $\lambda$-equivalent to $f\upharpoonright [a,b]$. We
define
$$C_*(f;[a,b]):=\{h\in C(f;[a,b])\colon~h(a)=f(a),~h(b)=f(b)\}.$$
\end{definition}

\begin{definition}\label{e:1}Let $f$ be from $C(\lambda)$ and $[a,b]\subset [0,1]$. For any fixed $m\in\N$, let us define the map $h=h\langle
f;[a,b],m\rangle\colon~[a,b]\to [0,1]$ by
$j\in\{0,\dots,m-1\}$ and
 \begin{equation}
h(a + x) := \begin{cases}
f\left (a+m \Big (x-\frac{j(b-a)}{m}\Big) \right )\text{ if } x\in \left [\frac{j(b-a)}{m},\frac{(j+1)(b-a)}{m} \right ],~j\text{ even}, \\
f\left (a+m \Big (\frac{(j+1)(b-a)}{m}-x \Big ) \right )\text{ if } x\in \left [\frac{j(b-a)}{m},\frac{(j+1)(b-a)}{m} \right ],~j\text{ odd}.
\end{cases}
\end{equation}
Then $h\langle f;[a,b],m\rangle\in C(f;[a,b])$ for each $m$ and $h\langle f;[a,b],m\rangle\in C_*(f;[a,b])$ for each $m$ odd .
\end{definition}

\begin{example}
In Figure \ref{fig:dendritemap1}, for  $f\in C(\lambda)$ shown on the left, on the right
we show the regular $3$-fold window perturbation of $f$ by $h=h\langle f;[a,b],3\rangle\in C_*(f;[a,b])$.
\end{example}

\begin{figure}[t]
\centering
\includegraphics[width=1\textwidth]{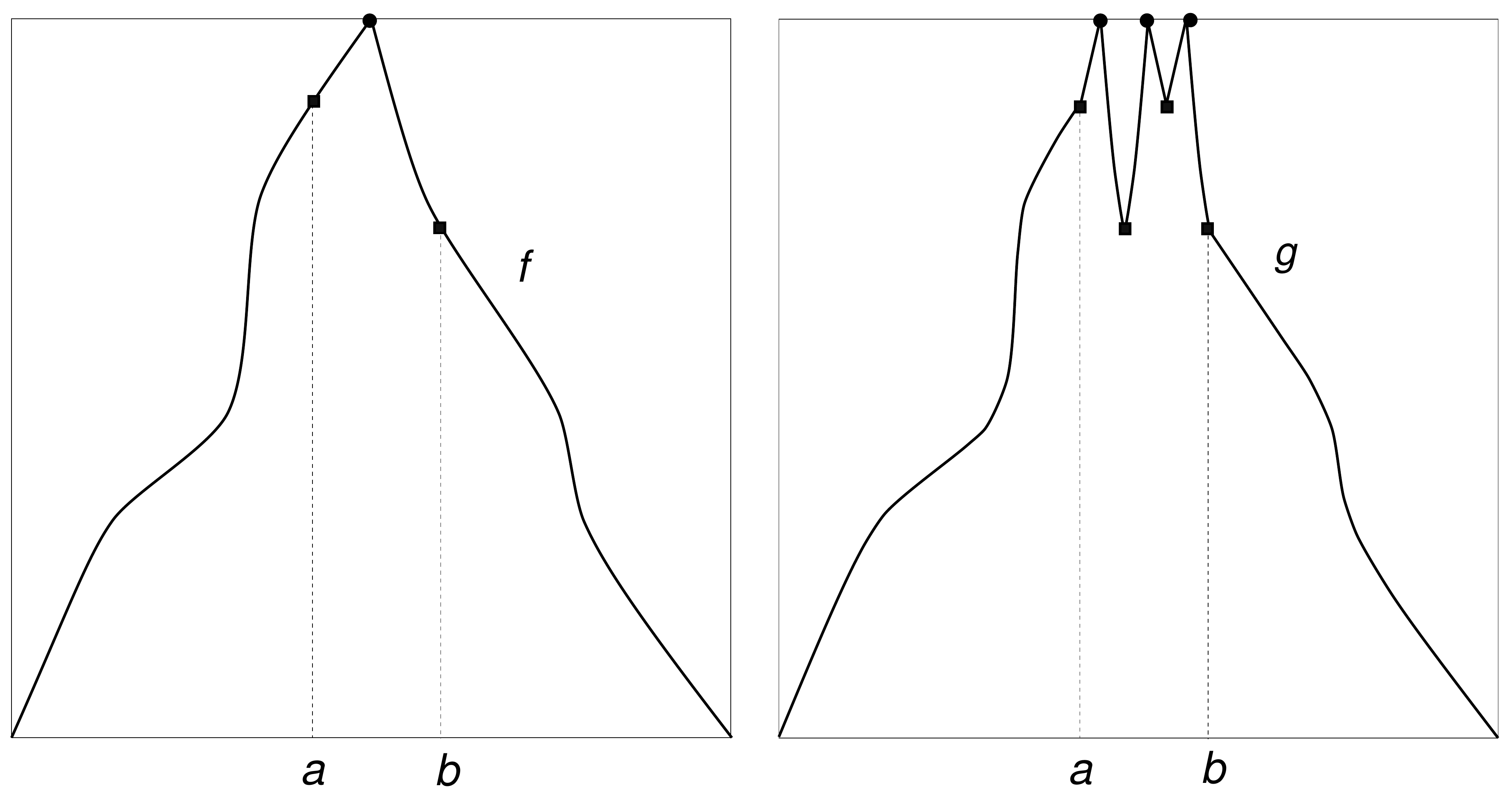}
\caption{The map $g$ is  $3$-fold regular window perburbation of the map $f$.}\label{fig:dendritemap1}
\end{figure}

For a fixed $h\in C_*(f;[a,b])$, the map $g=g\langle f,h\rangle\in C(\lambda)$ defined by
\begin{equation}\label{e:33}
g(x) := \begin{cases}
f(x)\text{ if } x\notin [a,b],\\
h(x)\text{ if } x\in [a,b]
\end{cases}
\end{equation}
will be called the \textit{window perturbation} of $f$ (by $h$ on $[a,b]$). In particular, if $h=h\langle[f;[a,b],m\rangle$, $m$ odd, resp. $h$ is piecewise
affine, we will speak of
\textit{regular $m$-fold, resp. piecewise affine window perturbation} $g$ of $f$ (on $[a,b]$) - see Figure \ref{fig:dendritemap1}.

The following useful observation will be repeatedly used in our text. We will omit its proof since
it is a  straightforward consequence of the uniform continuity of $f$.

\begin{lemma}\label{l:2}Let $f$ be from $C(\lambda)$. For each $\eps>0$ there is $\delta>0$ such that
\begin{equation}\forall~[a,b]\subset [0,1],~b-a<\delta~\forall~h\in C_*(f;[a,b])\colon~\rho(f,g\langle f,h\rangle)<\eps.\end{equation}
In particular, for each $\eps>0$ there is a positive integer $n_0>0$ such that for each $n>n_0$, if $I_j=[\frac{j}{n},\frac{(j+1)}{n}]$ and
$$g\upharpoonright I_j=h\langle f;I_j,m(j)\rangle$$ with odd numbers $m(j)$ for every
$j\in\{0,\dots,n-1\}$, then $\rho(f,g)<\eps$ independently of the numbers $m(j)$.
\end{lemma}

Below we introduce three classical types of mixing in  topological dynamics. We consider them in the context of $C(\lambda)$.

A map $f\in C(\lambda)$ is called
\begin{itemize}
\item \textit{transitive} if for each pair of nonempty open sets $U,V$, there is $n\ge 0$ such that $f^n(U)\cap V\neq\emptyset$,
\item \textit{topologicalyl mixing} if for each pair of nonempty open sets $U,V$, there is $n_0\geq0$ such that $f^n(U)\cap V\neq\emptyset$ for every $n\ge
    n_0$,
\item  \textit{leo} (\textit{locally eventually onto}) if for every nonempty open set $U$ there is $n\in{\mathbb N}$ such that $f^n(U)=[0,1]$.
\end{itemize}

For each map $f\in C(\lambda)$ the set of periodic points is dense in $[0,1]$, this is a consequence of the Poincar\'e Recurrence Theorem and the fact that in
dynamical system given by an
interval map the closures of recurrent points and periodic points coincide \cite{CoHe80}. Thus Proposition \ref{p:5} and Lemma \ref{l:1},
stated below, apply to elements of $C(\lambda)$.

               For two intervals $J_1,J_2$ with pairwise disjoint non-empty interiors we write $J_1<J_2$ if $x_1<x_2$ for some points $x_1\in J_1$ and $x_2\in
               J_2$. Throughout the article we
               will denote the interior of an interval $J$ by the notation $J^{\circ}$.

\begin{figure}[htb!!]
{\hspace*{-1.2cm}}\includegraphics[width=1.1\textwidth]{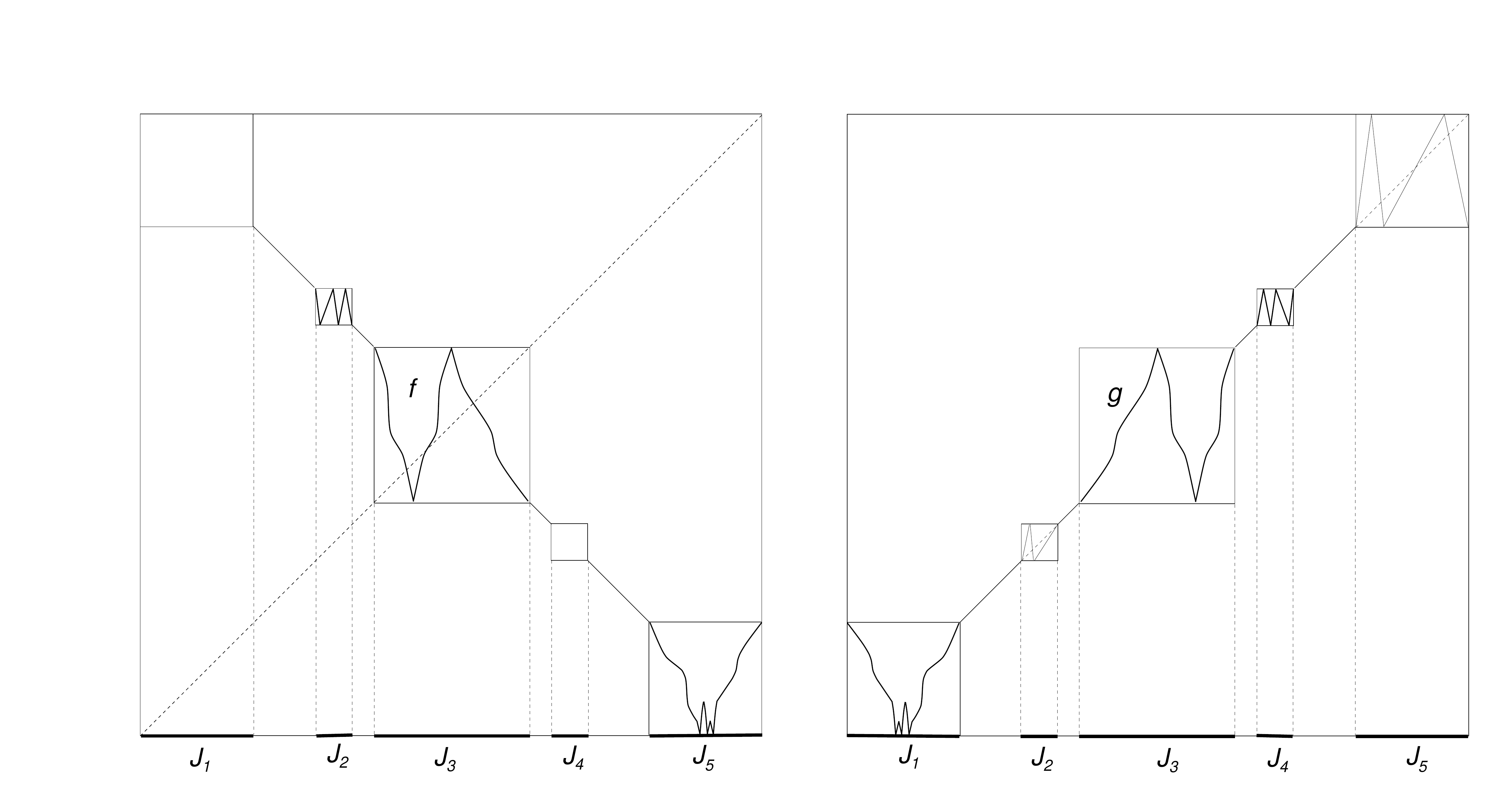}
\caption{$f,g\in C(\lambda)$; Prop. \ref{p:5}(\ref{e:28}): $\J(f)=\{J_i\}_{i=1}^5$, $f(J_1)=J_5$, $f(J_2)=J_4$, $f(J_3)=J_3$, $f(J_4)=J_2$, $f(J_5)=J_1$; Prop.
\ref{p:5}(\ref{e:18}):
$\J(g)=\{J_i\}_{i=1}^5$, $g(J_i)=J_i$ for each $i$. }\label{fig:dendritemap2}
\end{figure}

\begin{proposition}\label{p:5}\cite{BaMa85}~Suppose $f$ has a dense set of periodic points. The following  assertions hold (Figure \ref{fig:dendritemap2}).
\begin{itemize}
\item[(i)] There is a collection (perhaps finite or empty) $\J= \J(f) =\{J_1,J_2,\dots\}$ of
closed subintervals of $[0,1]$ with mutually disjoint interiors, such that for each $i$,  $f^2(J_i) = J_i$, and
there is a point $x_i\in J_i$ such that $\{f^{4n}(x_i)\colon~n\ge 0\}$ is dense in $J_i$.
\item[(ii)] If $\#\J\ge 2$ then either
\begin{equation}\label{e:18}\forall~J_1,J_2\in\J\colon~J_1<J_2\implies f(J_1)<f(J_2)\end{equation}
or
\begin{equation}\label{e:28}\forall~J_1,J_2\in\J\colon~J_1<J_2\implies f(J_1)>f(J_2).\end{equation}
\item[(iii)] For each $J\in\J$, $f(J)\in\J$ and $f^{-1}(f(J))=J$.
\item[(iv)] If (\ref{e:18}) is true then $f(J)=J$ for each $J\in\J$.
\item[(v)] If (\ref{e:28}) holds true $f(J)=J$ if and only if $J^{\circ}\cap \Fix(f)\neq\emptyset$ and there is at most one such interval.
\item[(vi)]If $x\in (0,1)\setminus \bigcup_{i\ge 1}J_i^{\circ}$, then $f^2(x)=x$.
 \item[(vii)] For each $J\in\J$, the map $f^2\upharpoonright J$ is topologically mixing.
   \item[(viii)] For each $J\in\J$, if $f(J)=J$ then the map $f\upharpoonright J$ is topologically mixing.
  \item[(ix)] The map $f$ is surjective.
\end{itemize}
 \end{proposition}

 \begin{proof} Properties (i)-(vi) had been proved in \cite{BaMa85}. The other ones easily follow. \newline
\noindent (vii) It is well known that an interval map $g\colon~J\to J$ is topologically mixing if and only if $g^2$ is transitive \cite[Theorem 46]{BlCo92}. By
(i), the set
$\{f^{4n}(x)\colon~n\ge 0\}$ is dense in $J$ for some $x\in J$, hence $f^4\upharpoonright J$ is transitive, i.e., $f^2\upharpoonright J$ is topologically
mixing.\newline
\noindent (viii) From (vii) we know that $f^2\upharpoonright J$ is topologically mixing hence also transitive and as in (vii) we can deduce that
$f\upharpoonright J\colon~J\to J$ is
topologically mixing.\newline
\noindent (ix) Since $f$ is continuous, the image $f([0,1])$ is a closed interval and by our assumption it contains a dense subset of $[0,1]$, so $f([0,1])=
[0,1]$.
 \end{proof}

\begin{lemma}\label{l:1} Suppose $f$ has a dense set of periodic points.
\begin{itemize}
\item[(i)] The map $f$ is transitive but not topologically mixing if and only if $\J=\{[0,b],[b,1]\}$ and $f([0,b])=[b,1]$.
\item[(ii)] The map $f$ is topologically mixing  if and only if $\J=\{[0,1]\}$.
 \item[(iii)] The map $f$ is leo if and only if $\J=\{[0,1]\}$ and both of the sets $f^{-2}(0)\cap (0,1)$ and $f^{-2}(1)\cap (0,1)$ are non-empty.
    \end{itemize}
    \end{lemma}
\begin{proof}

Parts (i) and (ii) follow immediately from Proposition \ref{p:5}, thus
we begin by the  only if direction of (iii).

By (ii), if  $\J \ne \{[0,1]\}$ then $f$ is not topologically mixing, and thus not leo, thus we suppose $\J = \{[0,1]\}$.
Suppose first that  $f^{-2}(0)\cap (0,1)=\emptyset$.
The leo map $f$ is continuous and surjective hence every point in $(0,1)$ must have at least one preimage in $(0,1)$. This fact and our assumption $f^{-2}(0)\cap
(0,1)=\emptyset$ imply
$f^{-1}(0)\subset \{0,1\}$.

 If $f(1)\neq 0$ then $f^{-1}(0)=\{0\}$ and $0\notin f^n((0,1])$ for every $n$ positive. If $f(1)=0$ then by our assumption, $f^{-1}(1)\cap (0,1)=\emptyset$
 hence $f(0)=1$. It implies that
 $f^2(0)=0$, $f^2(1)=1$ and $0\notin f^{2n}((0,1])$, what contradicts the leo property of $f$. The case when $f^{-2}(1)\cap (0,1)=\emptyset$ can be proven
 analogously.

We turn to the if direction. We assume that  $\J = \{[0,1]\}$ and that $f^{-2}(0)\cap (0,1)\neq\emptyset\neq f^{-2}(1)\cap (0,1)=\emptyset$; since by (ii) $f$ is
topologically mixing, for
every nonempty open $L\subset [0,1]$ there has to exist a positive $n$ for which $f^{n}(L)\cap f^{-2}(0)\neq\emptyset\neq f^{n}(L)\cap f^{-2}(1)$ hence
$f^{n+2}(L)=[0,1]$.
\end{proof}

For any set $X \subset C([0,1])$ we denote by $X_{property}$ the set of all maps in $X$ having a {\it property} (in lower index abbreviated) in question.
We denote by $PA(\lambda)$ the set of all piecewise affine maps from $C(\lambda)$.

\begin{proposition}\label{p:3}~The set $PA(\lambda)_{leo}$ is dense in $C(\lambda)$. \end{proposition}
\begin{proof}Fix an $f\in C(\lambda)$ and $\eps >0$. It had been shown in \cite{Bo91} that there exists a map $d^{\star}\in PA(\lambda)$ such that
$\rho(d^{\star},f)<\eps$.

Let us show with the help of Lemma \ref{l:2} that there exists a map $d^{\star\star}\in PA(\lambda)_{leo}$ for which $\rho(d^{\star\star},d^{\star})<\eps$. First
we prove\newline
\noindent {\bf Claim.}~$PA(\lambda)_{leo}=PA(\lambda)_{tmix}$.\newline
\noindent {\it Proof of Claim.}~Any leo map is topologically mixing. So let $f\in PA(\lambda)_{tmix}$ and show that $f\in PA(\lambda)_{leo}$. If $f^{-1}(0)\cap
(0,1)=f^{-1}(1)\cap
(0,1)=\emptyset$, then since $f$ is surjective either $f^{-1}(0)=\{0\}$, $f^{-1}(1)=\{1\}$ or $f^{-1}(0)=\{1\}$ and $f^{-1}(1)=\{0\}$. But $f\in C(\lambda)$, so
$f'\equiv 1$ on some
neighborhood of $\{0,1\}$ in the first case or $f'\equiv-1$ on some neighborhood of $\{0,1\}$ in the latter case - a contradiction with topological
mixing of $f$. So assume that
$f^{-1}(0)\cap (0,1)\neq\emptyset$. As in the proof of Lemma \ref{l:1}, since $f$ is continuous and surjective, $f^{-2}(0)\cap (0,1)\neq\emptyset$ and by Lemma
\ref{l:1}(iii), it is
sufficient to show that $f^{-2}(1)\cap (0,1)\neq\emptyset$, resp. $f^{-1}(1)\cap (0,1)\neq\emptyset$. Let $f^{-1}(1)\cap (0,1)=\emptyset$. We are done if
$\{0\}\subset f^{-1}(1)$, since then
$f^{-2}(1)\cap (0,1)\cap f^{-1}(0)\neq\emptyset$. It remains to comment the case $f^{-1}(1)=\{1\}$. Then since $f\in C(\lambda)$, $f'\equiv 1$ on some
neighborhood of $1$ - a contradiction
with topological mixing of $f$. The case when $f^{-1}(1)\cap (0,1)\neq\emptyset$ can be captured analogously.
This finishes the proof of the claim.

By our claim we are done if $d^{\star}\in PA(\lambda)_{tmix}$,
so assume that this is not the case.
Notice that since $d^{\star}$ is piecewise affine,
the set $\J(d^{\star})$ is a finite set, $\J(d^{\star})=\{J_i\colon~i=1,\dots,k\}$ with $2 \le k < \infty$, and thus
the set $[0,1]\setminus \bigcup\J(d^{\star})$ has a
finite number of connected components, each one is an interval.
Without loss of generality we can assume that each of these intervals is reduced to a single point
(if it were not the case,
we could use Proposition \ref{p:5} and a finite number of  regular $m$-fold, $m\ge 3$, piecewise affine window perturbations on a finite collection of
sufficiently small adjacent
subintervals of
those connected components as described in Lemma \ref{l:2} - see Figure \ref{fig:dendritemap3}(Left)).
Let $J=[a,b]$ and $J'=[b,c]$ be two adjacent element of $\J(d^{\star})$ - see Figure \ref{fig:dendritemap3}(Right).

\begin{figure}[htb!!]
\hspace*{.5cm}\includegraphics[width=.8\textwidth]{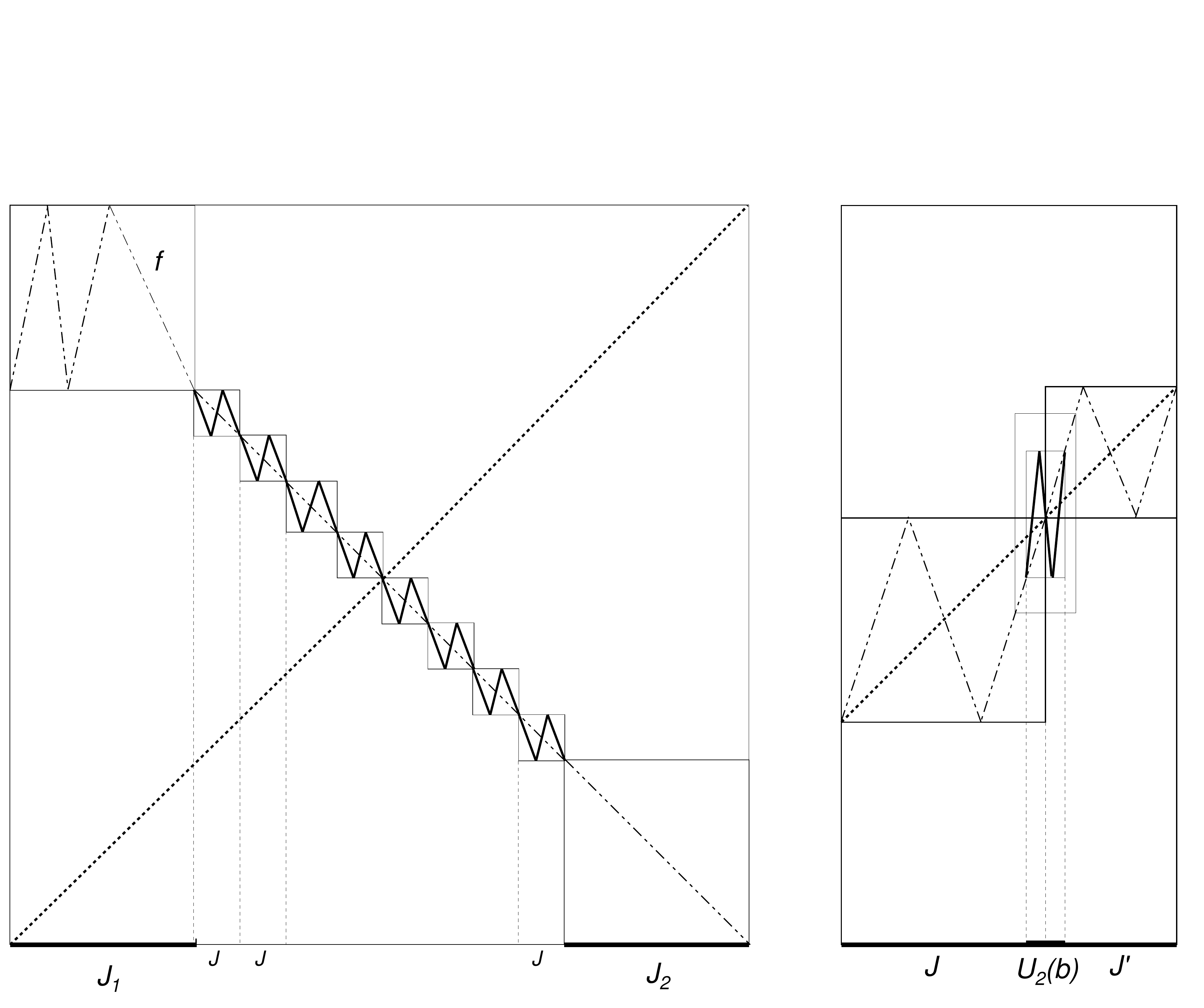}
\caption{Left: $f\in C(\lambda)$, $\bigcup\J(f)=\bigcup\{J_i\}_{i=1}^2$ is not dense, regular $3$-fold perturbations of $f$ on new $J$'s; Right: Perturbation on
$U_1(b)$ and $U_2(b)\subset
U_1(b)$ from the proof of Prop. \ref{p:3}.}\label{fig:dendritemap3}
\end{figure}

\begin{itemize}
\item Using a piecewise affine window perturbation (not necessarily regular) on some neighborhood of $b$ if necessary, w.l.o.g. we can assume that $d^{\star}$
    is strictly monotone with a
    constant slope on neighborhood $U_1(b)$ of $b$.
\item Choosing sufficiently small $\eps_1<\eps$ we can consider the regular $3$-fold window perturbation of $d^{\star}$ on $U_2(b)\subset U_1(b)$ resulting in
    the map $d^{\star}_1\in
    PA(\lambda)$ satisfying $\rho(d^{\star},d^{\star}_1)<\eps_1$. Moreover, by Proposition \ref{p:5} and Lemma \ref{l:1} either

\noindent - $\#\J(d^{\star}_1)=k-1<\#\J(d^{\star})=k$ in the case of Equation \eqref{e:18},  resp. Equation
\eqref{e:28} with $b\in\Fix(f)$ or

\noindent - $\#\J(d^{\star}_1)=k-2<\#\J(d^{\star})=k$ in the case of Equation \eqref{e:28} with $b\notin\Fix(f)$.

\end{itemize}

Finitely many modifications  of $d^{\star}$ with $\eps_1,\dots,\eps_{\ell}$, $\ell\le k-1$, satisfying $$\eps_1+\cdots +\eps_{\ell}<\eps$$ result to maps
$d^{\star}_1,\dots,
d^{\star}_{\ell}$ for which $$\rho(d^{\star}_{i},d^{\star}_{i+1})<\eps_{i+1}, \quad  i\in\{1,\dots,\ell-1\},$$  $d^{\star}_{\ell}\in C(\lambda)$,
$\#\J(d^{\star}_{\ell})=1$, and
$$\rho(d^{\star},d^{\star}_{\ell})<\eps_{1}+\sum_{i=1}^{\ell-1}
\rho(d^{\star}_{i},d^{\star}_{i+1})<\eps.$$

Summarizing, from Lemma \ref{l:1}(iii) we obtain that  $d^{\star\star}=d^{\star}_{\ell}$ is topologically mixing hence also from $PA(\lambda)_{leo}$ and
$\rho(d^{\star\star},f)<2\eps$.
\end{proof}

\begin{theorem}\label{t:4} The $C(\lambda)$-typical function is leo. \end{theorem}
\begin{proof}
By Proposition \ref{p:3} we can fix a countable dense collection $\{f_n\}_n$ from $PA(\lambda)_{leo}$.
Using a $2$-fold piecewise affine window perturbation of $f_n$ on $[0,\eps]$, resp. $[1-\eps,1]$ if necessary - see Example \ref{e:1} and Lemma \ref{l:2} -
without loss of generality we can
assume that for each $n \in \N$ we have $f_n(0) \in (0,1)$ and $f_n(1) \in (0,1)$.

Let $B(g,\varepsilon) := \{f \in C(\lambda): \rho(f,g) < \varepsilon\}$.
For a given sequence $\{ \varepsilon_n : \varepsilon_n > 0\}_n$ which we will choose later,
we consider the dense $G_\delta$ set
$$G :=  \bigcap_{N\ge 1}  \bigcup_{n \ge N} B(f_n,\varepsilon_n).$$
We claim that we can choose $\varepsilon_n$ in such a way that any $f \in G$ is leo.

Consider a sequence $(J_m)_{m \in \N}$ of all open rational subintervals of $(0,1)$.
For each $n,m$ there is a $j(n,m) \in \N$ such that $f_n^{j(n,m)}(J_m) = [0,1]$. Choose $\varepsilon_n > 0$ so small so that for all
$f \in B(f_n,\varepsilon_n)$ we have $f^{j(n,m)}(J_m) \supset (1/n, 1-1/n)$ for $m=1,2,\dots, n$. Additionally we assume that
$\varepsilon_n > 0$ is so small that $f(0) \in (0,1)$ and $f(1) \in (0,1)$.

Now consider an $f \in G$.  Then there exists an infinite sequence $(n_k)_{k \in \N}$ so that $f \in B(f_{n_k},\varepsilon_{n_k})$.
By Proposition \ref{p:5}(ix) the map $f$ is surjective.  Thus there are points $a,b \in [0,1]$ such that $f(a) = 0$ and $f(b) =1$.  By the choice
of $\varepsilon_n$ we have such points $a,b \in (0,1)$.

Fix an open  interval $J \subset [0,1]$.  Choose an $m$ so that $J_m \subset J$.  Suppose $n_k$ satisfies the following conditions
(i) $n_k \ge m$ and (ii) $a,b \in (1/n_k,1-1/n_k )$.  Assume $a < b$, the other case being similar.
By construction of $G$ and the above two assumptions we have $$f^{j(n_k,m)}(J) \supset f^{j(n_k,m)}(J_m) \supset (1/n_k,1-1/n_k) \supset [a,b].$$
 Thus $f^{j(n_k,m)+1}(J) \supset f([a,b]) = [0,1].$
\end{proof}

For  integers $a \ge b \ge 0$ let $f^{[a,b]}(x) := \{f^j(x): a \le j
\le b\}$.  A family of orbit segments $\{f^{[a_j,b_j]}(x_j)\}_{j=1}^n$
is an \textit{$N$-spaced specification} if $a_i - b_{i-1} \ge N$ for $2 \le i \le n$.
We say that a specification $\{f^{[a_j,b_j]} (x_j)\}_{j=1}^n$ is \textit{$\varepsilon$-shadowed} by $y  \in [0,1]$ if
$$d(f^k(y),f^k(x_i)) \le \varepsilon \text{ for } a_i \le k \le b_i \text{ and } 1 \le i \le n.$$

We say that $f$ has the \textit{specification property} if for any  $\varepsilon > 0$
there is a constant $N = N(\varepsilon)$  such that any $N$-spaced specification
$\{f^{[a_j,b_j]} (x_j)\}_{j=1}^n$ is $\varepsilon$-shadowed by some $y \in [0,1]$.
If additionally, $y$ can be chosen in such
a way that $f^{b_n-a_0 +N}(y) = y$ then $f$ has the \textit{periodic specification property}.

Applying a result of Blokh \cite{Bl} we obtain
\begin{corollary}\label{blokh}
 The $C(\lambda)$-typical function satisfies the periodic specification property.
\end{corollary}

\section{Mixing properties in $C(\lambda)$}\label{s:3}
We start by introducing three classical types of mixing in a measure-theoretical dynamics \cite{Wa82}. We state them in the context of $C(\lambda)$.

\begin{definition}A map $f\in C(\lambda)$ is called
\begin{itemize}
\item[(i)] \textit{ergodic}, if for every $A,B\in\B$,
$$\lim_{n\to\infty}\frac{1}{n}\sum_{j=0}^{n-1}\lambda(f^{-j}(A)\cap B)=\lambda(A)\lambda(B).
$$
\item[(ii)] \textit{weakly mixing}, if for every $A,B\in\B$,
$$\lim_{n\to\infty}\frac{1}{n}\sum_{j=0}^{n-1}\vert \lambda(f^{-j}(A)\cap B)-\lambda(A)\lambda(B)\vert=0.
$$
\item[(iii)] \textit{strongly mixing}, if for every $A,B\in\B$,
$$\lim_{n\to\infty}\lambda(f^{-j}(A)\cap B)=\lambda(A)\lambda(B).
$$

\end{itemize}
\end{definition}

Analogously as before, for a subset $X\subset C(\lambda)$ we denote by $X_{slope>1}$ the set of all maps $f$ from $X$ for which $\vert f'(x)\vert>1$
for all $x\in [0,1]$ at which derivative of $f$ exists.

We denote by $PAM(\lambda)$ the set of all piecewise affine Markov maps in $PA(\lambda)$, i.e., maps for which all points of discontinuity of the derivative and
also both endpoints $0,1$ are
eventually periodic.

\begin{proposition}\label{p:6}~The set $PAM(\lambda)_{leo}$ is dense in $C(\lambda)$. \end{proposition}

\begin{proof}Let $f\in PA(\lambda)_{leo}$, fix $\eps>0$. Denote $R(f)$ the set containing $\{0,1\}$ and all points of discontinuity of the derivative of $f$, let
$S(f)\subset R(f)$ be
eventually periodic points from $R(f)$ and $T(f)=R(f)\setminus S(f)$.

Clearly the set $R(f)$ is finite. Fix $t\in T=T(f)$. By Proposition \ref{p:5}(ix) and Lemma \ref{l:2} we can consider a periodic orbit $P=\{p_1<\cdots<p_k\}$ of
$f$ such that for some three
consecutive points $p_{i-1}<p_i<p_{i+1}$
\begin{itemize}
\item $f\upharpoonright [p_{i-1},p_{i+1}]$ is affine,
\item $\orb(S(f),f)\cap [p_{i-1},p_{i+1}]=\emptyset$,
\item every window perturbation of $f$ by $h\in C(f;[p_{i-1},p_{i+1}])$ on $[p_{i-1},p_{i+1}]$ is $\eps/m$-close to $f$, where $m=\#T$,
\item every piecewise affine window perturbation of $f$ on $[p_{i-1},p_{i+1}]$ belongs to $PA(\lambda)_{leo}$,
\item $\orb(t,f)\cap (p_{i-1},p_{i+1})\neq\emptyset$.
\end{itemize}
Let $f^{\ell}(t)$ be the first iterate of $t$ in $(p_{i-1},p_{i+1})$. By Lemma \ref{l:2} there exists a $5$-fold piecewise affine window perturbation (not
necessarily regular) $g_1$ of $f$ by
$h$ on $[p_{i-1},p_{i+1}]$ satisfying
$$g_1(f^{\ell}(t))=g_1(p_i)=f(p_i).$$

Then $\#R(g_1)=\#R(f)+6$ and $\#S(g_1)\ge \#S(f)+7$ hence
\begin{align}&\#T(f)-1=m-1=\#R(f)+6-(\#S(f)+7)\ge \nonumber\\
&\ge\#R(g_1)-\#S(g_1)=\# T(g_1).\nonumber
\end{align}
 Repeating the above procedure maximally $m=\#T$-times, we obtain the required Markov map $g\in PAM(\lambda)_{leo}$.
\end{proof}

Let $f$ be from $PAM(\lambda)_{slope>1}$ with a Markov partition
\begin{equation*}\label{e:22}\A=\{A_0=[x_0,x_1]\le \cdots\le A_{N-1}=[x_{N-1},x_N]\},\end{equation*} where the set $P_f=\{0=x_0<\cdots<x_N=1\}$ contains orbits
of all points of discontinuity
of derivative of $f$ and of the endpoints. To each point $x\in [0,1]$ we associate its itinerary  $\Phi(x)=(\phi_i(x))_{i\ge 0}$ with respect to $\A$, i.e.,
$\phi_i(x)\in \{0,1,\dots,N-1\}$
and $f^i(x)\in A_{\phi_i(x)}$ for each $i\ge 0$ (in this settings $\Phi$ is a one-to-finite multivalued map). Since $f$ is continuous, the system
$(\Phi([0,1]),\sigma)$ is a subshift of the
full shift $(\{0,1,\dots,N-1\}^{\N_0},\sigma)$ on the symbols $\{0,1,\dots,N-1\}$ \cite{Wa82}.

Any map from $PAM(\lambda)_{leo}$ satisfies the hypothesis of \cite[Theorem 3.2]{ADU93}. So any such map is in fact exact, i.e., for every $A\in \bigcap_{n\ge 0}
T^{-n}(\B)$,
$\lambda(A)\lambda(A^c)=0$. It is known that every exact map has one-sided countable Lebesgue spectrum and hence is strongly mixing \cite[p. 115]{Wa82}. For our
purpose it will be convenient
to prove explicitly the following.

\begin{lemma}\label{l:4}Let $f$ be from $PAM(\lambda)_{slope>1}$, consider $\A$ and $\Phi$ as above. The system $([0,1],\B,\lambda,f)$ is isomorphic to the
one-sided Markov shift
$(\Phi([0,1])),\B',\mu,\sigma)$, where the measure $\mu$ on the Borel $\sigma$-algebra $\B'$ is given by the probability vector
\begin{equation}\label{v:1}p=(\lambda(A_0),\dots,\lambda(A_{N-1}))\end{equation} and the stochastic matrix $P=\left (p_{ij}\right )_{i,j=0}^{N-1}$, where
    \begin{equation}\label{m:1}
p_{ij}=\begin{cases}\frac{\lambda(A_j)}{\lambda(f(A_i))},&f(A_i)\supset A_j \\ 0,&\text{otherwise}\end{cases}.
\end{equation}
In particular every map from $PAM(\lambda)_{leo}$ is strongly mixing.
\end{lemma}
\begin{proof}For the definition of isomorphic measure theoretic systems see \cite[Definition 2.4]{Wa82}. Clearly, the vector $p$ is a probability vector and,
since $f\in C(\lambda)$ with
constant derivative on each $A_i$, the matrix $P$ is stochastic and $pP=p$. So the measure $\mu$ is defined well on the Borel $\sigma$-algebra $\B'$ generated by
the cylinders in
$\Phi([0,1])$.  Since $f\in PAM(\lambda)_{slope>1}$, $\Phi$ is injective. Let $N$ be the set of those points $x$ from $[0,1]$ for which the set $\Phi(x)$
consists of more itineraries. Then
$N$ is countable and, since $\Phi$ is a one-to-finite multivalued map, the set $\Phi(N)$ is also countable. Hence $\Phi\colon~[0,1]\setminus N\to
\Phi([0,1]\setminus N)$ is a bijection and
$\lambda([0,1]\setminus N)=\mu(\Phi([0,1]\setminus N))=1$. Obviously,
$$\Phi\circ f=\sigma\circ\Phi\text{ on }[0,1]\setminus N.
$$
To finish the proof we need to show that $$\lambda(\Phi^{-1}(A))=\mu(A)\text{ for each }A\in\B'.$$
Obviously it is sufficient to verify the last equality for cylinders, i.e., the sets $$C_{\phi_0,\dots,\phi_{k-1}}=\{(\phi_i(x))_{i\ge 0}\in
\phi([0,1])\colon~\phi_0(x)=\phi_0,\dots,\phi_{k-1}(x)=\phi_{k-1}\},$$
where $k\in\N$ and $\phi_0,\dots,\phi_{k-1}\in \{0,\dots,N-1\}$. By the definition of the Markov shift
$$\mu(C_{\phi_0,\dots,\phi_{k-1}})=\lambda(A_{\phi_0})\prod_{j=1}^{k-1}\frac{\lambda(A_{\phi_j})}{\lambda(f(A_{\phi_{j-1}}))}=\clubsuit,$$
where the second factor equals to one if $k=1$. Since $f$ has a constant derivative on each $A_i$,
$$\clubsuit=\lambda(\Phi^{-1}(C_{\phi_0,\dots,\phi_k})).
$$
If $f\in PAM(\lambda)_{leo}\cap PAM(\lambda)_{slope>1}$, the matrix $P$ is irreducible and aperiodic,
hence $$(\Phi([0,1])),\B',\mu,\sigma)$$ is strogly mixing \cite[Theorem 1.31]{Wa82}. By the previous, it
is also true for isomorphic $([0,1],\B,\lambda,f)$.
\end{proof}

\begin{corollary}\label{p:4}The set $C(\lambda)_{smix}$ of strongly mixing maps is dense in $C(\lambda)$.\end{corollary}
\begin{proof}It is a consequence of Proposition \ref{p:6} and Lemma \ref{l:4}.
\end{proof}

\begin{theorem}\label{t:5}$C(\lambda)$-typical function is weakly mixing. \end{theorem}

\begin{proof}
By Proposition \ref{p:6}  and Lemma \ref{l:4} we can consider a countable dense set $\{f_n\}_n$ of weakly mixing maps.
Suppose $\varepsilon_n$ are strictly positive. Let
$$\mathcal G:=\bigcap_{N \ge 1}\bigcup_{n \ge N} B(f_n,\varepsilon_n).$$
Clearly $\mathcal G$ is a dense $G_\delta$.
We will show that the $\varepsilon_i$ can be chosen in such a way
that all the configurations in $\G$ are weakly mixing.

Let $\{h_j\}_{j \ge 1}$ be a countable, dense
collection  of continuous functions in $L^1(X \times X)$.
For any $f \in C(\lambda)$ and $\ell \ge 1$, let
$$S^f_{\ell}h_j(x,y)  := \frac{1}{\ell} \sum_{k=0}^{\ell-1} h_j \big( (f \times f)^k(x,y)\big).$$
The map $f$ is weakly mixing if and only if the map $f \times f$ is ergodic, and
by the Birkhoff ergodic theorem, the map $f \times f$ is ergodic  if and only if  we have
$$\lim_{\ell \to \infty} S^f_\ell h_j(x)  = \int_{X \times X} h_j(s,t) \, d(\lambda(s) \times \lambda (t))$$
for all $j \ge 1$.

For each $n$ since  $f_n$ is weakly mixing, there exists a set $B_n \subset X \times X$  and a positive integer $\ell_n$ such that
$\lambda (B_n) > 1 - \frac1i$ and
$$\Big |S^{f_n}_{\ell_n} h_j(x,y) - \int_{X \times X} h_j(s,t) \,  d(\lambda(s) \times \lambda (t)) \Big | < \frac1i$$
for all $(x,y) \in B_n$, $1 \le j \le n$. We can assume that $\lim_{n \to \infty} \ell_n = \infty$.

Now we would like to extend these estimates to the neighborhood $B(f_n,\varepsilon_n)$
for a sufficiently small strictly positive $\varepsilon_n$.
By the triangular inequality we have:
$$\begin{array}{ll}
 \Big |S^{g}_{\ell_n} h_j(x,y)   -  \int_{X \times X} h_j(s,t) \,  d(\lambda(s) \times \lambda (t)) \Big |    \le\\
 \Big |S^{g}_{\ell_n} h_j(x,y)  - S^{f_n}_{\ell_n} h_j(x,y) \Big |  +  \Big |S^{f_n}_{\ell_n} h_j(x,y)  -   \int_{X \times X} h_j(s,t) \,  d(\lambda(s) \times
 \lambda (t)) \Big |
 .\end{array}$$
For any point $(x,y)$, and any $\ell \ge 1$  the point $g^{\ell} (x,y) $ varies continuously with $g$ in a small neighborhood of $f_n$; thus
we can find $\varepsilon_n>0$,  and a set $\hat B_{n} \subset B_{n}$
of measure larger than $1-\textstyle\frac2i$
so that if $g \in B(f_n,\varepsilon_n)$, then
$$\Big |S^{g}_{\ell_n} h_j(x,y)  -  \int_{X \times X} h_j(s,t) \,  d(\lambda(s) \times \lambda (t)) \Big |    < \frac2i$$
for all $(x,y)  \in \hat B_{n}$,  $1 \le j\le i$.

For each  $g \in \G$
there is an infinite sequence $n_k$ such that $g \in B(f_{n_k},\varepsilon_{n_k})$.  Consider
$\mathcal{B}(g)  = \bigcap_{M=1}^\infty \bigcup_{i=M}^\infty \hat B_{n_k}$.
Since $\lambda(\hat B_{n_k}) > 1 - \frac1{n_k}$, it follows that  $\lambda(\mathcal{B}(g)) = 1$.

We can thus conclude that for  $\lambda$-a.e.~$(x,y)$, for all $j \ge 1$,
\begin{equation}\label{e1}
\lim_{k \to \infty} S^{g}_{\ell_n} h_j(x,y)  =  \int_{X \times X} h_j(s,t) \,  d(\lambda(s) \times \lambda (t)),
\end{equation}
and thus $g$ is weakly mixing.
\end{proof}

\begin{definition}\label{d:2}We say a piecewise monotone map $f\colon~[0,1]\to [0,1]$ is expanding
if there is a constant $c > 1$ such that $\vert f(x)-f(y)\vert > c\vert x-y\vert$
whenever $x$ and $y$ lie in the same monotone piece. If $f$ is expanding Markov
and a finite set $P_f=\{ x_0 < \cdots< x_N\}$ contains orbits of all points of discontinuity of the derivative and also of the endpoints $0,1$, we let $P^* =
\{0,\dots,N\}$ and define
$f^*\colon~P^* \to P^*$ by
$f^*(i)=j$ if $f(x_i)=x_j$.
\end{definition}

\begin{remark}\label{r:1}The set $P$ from Definition \ref{d:2} is not uniquely determined. Any set $P'=\bigcup_{k=0}^nf^{-k}(P)$, $n\in\mathbb N$, is also a
finite set that contains orbits of
all points of discontinuity of the derivative and also of the endpoints $0,1$.
\end{remark}

\begin{theorem}\cite[Theorem 2.1]{BlCo87}\label{t:13}~Expanding Markov maps $f$ and $g$ are topologically conjugate via an increasing homeomorphism $h$ if and
only if $f^*=g^*$. In this case
$h(P_f)=P_g$, where $g=h\circ f\circ h^{-1}$.
\end{theorem}

The next part of this paragraph will be devoted to the strong mixing maps in $C(\lambda)$. We start with one useful lemma.

\begin{lemma}\label{l:3}Let $f$ be from $PAM(\lambda)_{leo}$.  For each $\eps>0$ there exists a strongly mixing measure $\mu\neq\lambda$ preserved by the map $f$
and a homeomorphism
$h\colon~[0,1] \to [0,1]$ such that for $\nu=(\mu+\lambda)/2$
  $$\lambda=h^*\nu,~\text{ i.e., }g=h\circ f\circ h^{-1}\in C(\lambda)\text{ and }\vert\vert f-g\vert\vert<\eps.
  $$
\end{lemma}
\begin{proof} Consider the Markov partition
$$\A=\{A_0=[x_0,x_1]\le \cdots\le A_{N-1}=[x_{N-1},x_N]\}$$ for $f$, where the set $P_f=\{0=x_0<\cdots<x_N=1\}$ contains all orbits of points of discontinuity of
derivative of $f$ and of the
endpoints. Using Definition \ref{d:2} and Remark \ref{r:1} we can assume that for some $\{x_{i-1}<x_i<x_{i+1}\}\subset P_f$  there are points
$\{x_{\ell-1}<x_{\ell}<x_{\ell+1}\le
x_{r-1}<x_{r}<x_{r+1}\}\subset P_f$ such that
$$f(\{x_{\ell-1},x_{\ell+1}\})=f(\{x_{r-1},x_{r+1}\})=\{x_{i-1},x_{i+1}\},~f(x_{\ell})=f(x_r)=x_i$$
and
\begin{equation}\label{e:19}f^{-1}(x_{\ell})\cap P_f=\emptyset=f^{-1}(x_{r})\cap P_f.\end{equation}

The last conditions in (\ref{e:19}) imply that for every $A_j\in\A$ and $s\in\{\ell,r\}$

\begin{equation}\label{e:20}\emptyset\neq f(A_j)\cap (A_{s-1}\cup A_{s})^{\circ}\implies f(A_j)\supset A_{s-1}\cup A_{s},\end{equation}
where as before $J^{\circ}$ denotes the interior of an interval $J$.

In what follows we introduce a map $\alpha$ from $PAM(\lambda)_{leo}$ such that $P_{\alpha}$ differs from $P_f$ only in the points $x_{\ell},x_r$. Fix
$\delta>0$. We can consider $\delta_1\in
(0,\delta)$ and
$$P_{\alpha}=\{y_0<\cdots<y_{\ell-1}<y_{\ell}<y_{\ell+1}\le y_{r-1}<y_{r}<y_{r+1}<\dots <y_N\}
$$
satisfying
\begin{itemize}
\item $x_i=y_i$ for $i\notin \{\ell,r\}$ and $0<\vert x_{\ell}-y_{\ell}\vert<\delta_1$, $0<\vert x_{r}-y_{r}\vert<\delta_1$,
\item $\tilde\alpha(y_i)=y_j$ if and only if $f(x_i)=x_j$
\item the connect-the-dots map $\alpha$ extending $\tilde\alpha$ from the set $P_{\alpha}$ to the whole interval $[0,1]$ satisfies $\alpha\in
    PAM(\lambda)_{leo}$.
\end{itemize}
Since both maps $f$ and $\alpha$ are expanding, by (\ref{e:20}) also Markov and $f^*=\alpha^*$, from Theorem \ref{t:13} we obtain that $\alpha=h_1\circ f\circ
h_1^{-1}$ with
$h_1(P_f)=P_{\alpha}$. By Remark \ref{r:1} we can consider the set $P_f$ $\delta$-dense in $[0,1]$ hence the homeomorphism $h_1$ fulfils $0<\vert\vert
h_1-\id\vert\vert<2\delta$.

 By Lemma \ref{l:4} the map $\alpha$ with respect to $\lambda$ is measure isomorphic to a one-sided Markov shift given by the probability vector
 $q=(\lambda([y_0,y_1]),\dots,\lambda([y_{N-1},y_N]))$ and the stochastic matrix $Q=\left (q_{ij}\right )_{i,j=0}^{N-1}$, where
    \begin{equation}\label{m:11}
q_{ij}=\begin{cases}\frac{\lambda([y_{j-1},y_j])}{\lambda(\alpha([y_{i-1},y_i]))},&\alpha([y_{i-1},y_i])\supset [y_{j-1},y_j] \\ 0,&\text{otherwise}\end{cases}.
\end{equation}

Since $\alpha=h_1\circ f\circ h_1^{-1}$, the measure $\mu\neq\lambda$ (for $h_1(y_{\ell})\neq x_{\ell}$) given by $\lambda=h_1^*\mu$ is a strongly mixing measure
preserved by the map $f$. It
follows that the measure $\nu=\frac{\mu+\lambda}{2}$, as a convex combination of two strongly mixing measures, is a nonergodic measure with $\supp~\nu=[0,1]$ and
preserved by the map $f$. Let
us consider a homeomorphism $h\colon~[0,1]\to [0,1]$ defined by
$$\lambda=h^*\nu.
$$
Then from $\vert h_1(x)-x\vert<2\delta$ fulfilling for each $x\in [0,1]$ we obtain
$$
x-\delta<\nu([0,x])=h(x)=\frac{\mu([0,x])+x}{2}=\frac{h_1(x)+x}{2}<x+\delta,
$$
i.e., $\vert\vert h-\id\vert\vert<\delta$. Now, taking $\delta$ sufficiently small we obtain
 $$g=h\circ f\circ h^{-1}\in C(\lambda)\text{ and }\vert\vert f-g\vert\vert<\eps.
  $$
\end{proof}

\begin{theorem}\label{t:8}The set of all strongly mixing maps in $C(\lambda)$ is of the first category.\end{theorem}
\begin{proof}As before we denote $C(\lambda)_{smix}$ the set of all strongly mixing maps in $C(\lambda)$. Then
\begin{align}\label{e:21}C(\lambda)_{smix}=\bigcap_{\eps>0}\bigcap_{A,B\in\B}\bigcup_{n\ge 1}\bigcap_{k\ge n}F_{\eps,A,B,k}
\end{align}
where
$$F_{\eps,A,B,k}=\{f \in C(\lambda) \colon~\vert\lambda(f^{-k}(A)\cap B)-\lambda(A)\lambda(B)\vert\le \eps\}.
$$
is a closed set for each pair $A,B\in\B$. Using Propositions \ref{p:6} and Corollary \ref{p:4} we can consider a dense sequence $\{f_j\}_j$ of piecewise affine
leo, strongly mixing maps in
$C(\lambda)$.

For a positive sequence $\{\eps_m\}_{m}$ converging to $0$ and a map $f_j$ let us denote $\mu_{j,m}$, $h_{j,m}$, $\nu_{j,m}=(\mu_{j,m}+\lambda)/2$ all objects
guaranteed in Lemma
\ref{l:3}(ii) for $f=f_j$ and $\eps=\eps_m$.  Since each $\mu_{j,m}$ is orthogonal to $\lambda$, there is a Borel set $A_{j,m}$ satisfying $\lambda(A_{j,m})=1$
and $\mu_{j,m}(A_{j,m})=0$. Put
$A=\bigcap_{j,m}A_{j,m}$. Then $\lambda(A)=1$ and we can write for the map  $g_{j,m}=h_{j,m}\circ f_j\circ h_{j,m}^{-1}\in C(\lambda)$ and each $k\in\N$
\begin{align}&
\vert\lambda(g_{j,m}^{-k}(h_{j,m}(A))\cap h_{j,m}(A))-\lambda(h_{j,m}(A))\lambda(h_{j,m}(A))\vert=\nonumber\\
=&\vert\nu_{j,m}(f_j^{-k}(A)\cap A)-\nu_{j,m}(A)\nu_{j,m}(A)\vert=\vert\frac12\lambda(A)-\frac12\lambda(A)\frac12\lambda(A)\vert=\frac14.
\end{align}
It shows that the closed set $F_{1/5,A,A,k}$ is nowhere dense for each $k$ hence by (\ref{e:21}) the set $C(\lambda)_{smix}$ is of the first category in
$C(\lambda)$.
\end{proof}

The following theorem states a general result analogous to one of V. Jarn\'\i k \cite{Ja33}.  Recall that by  {\it a knot point} of function $f$ we mean a point
$x$ where
$D^{+}f(x)=D^{-}f(x)=\infty$ and $D_{+}f(x)=D_{-}f(x)=-\infty$.

\begin{theorem}\label{t:1}\cite{Bo91}~$C(\lambda)$-typical function has a knot point at $\lambda$-almost every point.
\end{theorem}

\begin{corollary} \label{cor:onto} The $C(\lambda)$-typical function maps a  set of Lebesgue measure zero  onto $[0,1]$. \end{corollary}
\begin{proof}  Let $K$ be the set of knot points of $f$.  Each level set contains its maximum, it cannot be a knot point.
Thus $f(K^c) = [0,1]$. \end{proof}

It is an interesting question if functions from $C(\lambda)$ with knot points $\lambda$-almost everywhere have infinite topological entropy.

\section{Metric entropy in $C(\lambda)$}
We start this section by an easy application of the Rohlin entropy formula
(see for example Theorem 1.9.7 in \cite{PrUr10}).
\begin{lemma}\label{l:5}Let $f$ be from $PA(\lambda)$. Then
$$h_{\lambda}(f)=\int_{0}^1\log\vert f'(x)\vert~\d\lambda(x).
$$
\end{lemma}

It follows from \cite[Corollary 4.14.3]{Wa82} that if $f\in C(\lambda)\setminus\{\id,1-\id\}$ then $h_{\lambda}(f)>0$. Analogously as before, for $c\in
(0,\infty]$ and $X\subset C(\lambda)$
we denote by $X_{entr<c}$, resp.  $X_{entr=c}$ the set of all maps $f$ from $X$ for which $h_{\lambda}(f)<c$, resp. $h_{\lambda}(f)=c$.

\begin{proposition}
For every $c\in (0,\infty)$ the set $PAM(\lambda)_{entr=c}$ is dense in $C(\lambda)$.
\end{proposition}
\begin{proof} We claim  that for each $\eps>0$
\begin{equation}\label{e:25}\forall~f\in PA(\lambda)_{slope>1}~\forall~ \delta>0 \colon~B(f;\delta)\cap PA(\lambda)_{entr<\eps}\cap
PA(\lambda)_{slope>1}\neq\emptyset.\end{equation}

In order to verify (\ref{e:25}) we will proceed in several steps. In the first step we show that in $C(\lambda)$ any piecewise affine map with full
laps can be approximated by a piecewise affine map with exactly two distinct slopes and of arbitrarily small metric entropy;  in the second step we
generalize our construction to any piecewise affine map. In the third step we prove the statement of the proposition.

\noindent {\bf I.} Let $F\colon~[0,1]\to [0,1]$ be a continuous piecewise affine map with $m>1$ full laps, i.e., for which there are points
$0=x_0<x_1<\cdots<x_m=1$ such that
$F\upharpoonright [x_i,x_{x_{i+1}}]$, $i=0,\dots,m-1$, is affine and $F([x_i,x_{x_{i+1}}])=[0,1]$ for each $i$. Clearly $F\in PAM(\lambda)_{slope>1}$ and $\vert
F'(x)\vert=1/\alpha_i$ for
$\alpha_i=\lambda([x_i,x_{i+1}])$ and each $x\in (x_i,x_{i+1})$. In fact the map $F$ is uniquely determined by the $(m+1)$-tuple
$(\pm,\alpha_0,\dots,\alpha_{m-1})$ (we write $F\sim
(+,\alpha_0,\dots,\alpha_{m-1})$) satisfying
\begin{equation}\label{e:*}
\sum_{i=0}^{m-1}\alpha_i=1,
\end{equation} $\alpha_i>0$ for each $i$ and in which the first coordinate indicates if $F$ increases ($+$), resp. decreases ($-$) on the interval
$[0,\alpha_0]$. Let us assume that
$\alpha_i=p_i/q\in\Q$ for each $i$ and for $\eta>0$ and integer $M>2$ put
$$ r(\eta,M)=\frac{\eta}{M(m-1)}\text{ and }s(\eta,M)=\frac{1-\eta}{M}.
$$

 For $\eta\in (0,1)$, an $M$ divisible by $q$ where $\alpha_i=p_i/q$, for $i\in\{0,\dots,m-1\}$  define a continuous map
 $h_i=h_{i}[\eta,M]\colon~[0,Mq^{-1}\gamma_i]\to \mathbb R$, where
 $r=r(\eta,M)$, $s=s(\eta,M)$, $\gamma_i=\gamma_{i}(\eta,M)=p_is+(q-p_i)r$ and
\begin{itemize}

\item $h_i$ is affine with slope $\frac{1}{1-\eta}$ on $[q_ir,q_ir+p_is]$, where $q_i=\sum_{j\le i-1}p_j$
\item $h_i$ is affine with slope $\frac{m-1}{\eta}$ on $[0,q_ir]$ and $[q_ir+p_is,\gamma_i]$
\item $h_i(x)=h_i(x-(\ell-1)\gamma_i)+\frac{(\ell-1)q}{M}$ for $x\in [(\ell-1)\gamma_i,\ell\gamma_i]$, $1\le\ell\le Mq^{-1}$
\item $h_i(0)=0$.
\end{itemize}

We leave the straightforward verification of the following properties to the reader (see Figure \ref{fig:dendritemap5}).

\begin{itemize}
\item[(i)] $h_i(Mq^{-1}\gamma_i)=1$
\item[(ii)] $h_i$ is strictly increasing
\item[(iii)] $h_i$ is a piecewise affine map with two slopes   $\frac{m-1}{\eta}$ and $\frac{1}{1-\eta}$, the latter one on $Mq^{-1}$ pairwise disjoint closed
    intervals
\item[(iv)] $\lim_{\eta\to 0_+}Mq^{-1}\gamma_i(\eta,M)=\alpha_i$ and
$$\forall~\iota,\kappa>0~\exists~\eta',M'~\forall~\eta<\eta',M>M'\colon~\max_{x\in [x_i+\iota,x_{i+1}-\iota]}\vert F(x)-h_{i}[\eta,M](x)\vert<\kappa.
$$
\item[(v)] $Mq^{-1}\sum_{i=0}^{m-1}\gamma_i(\eta,M)=1$ for each pair $\eta,M$
\end{itemize}
Let   $H=H[\eta,M]\colon~[0,1]\to [0,1]$ be defined by (we put $\beta_i=Mq^{-1}\gamma_i(\eta,M)$)
\begin{equation*}\label{ap:1}
H(x) :=\begin{cases}h_i(x-\sum_{j\le i-1}\beta_j),\text{ for }x\in [\sum_{j\le i-1}\beta_j,\sum_{j\le i}\beta_j],~i\text{ even }\\ h_i(\sum_{j\le
i}\beta_j-x),\text{ for }x\in [\sum_{j\le
i-1}\beta_j,\sum_{j\le i}\beta_j],~i\text{ odd. }\end{cases}
\end{equation*}

Clearly $H\in PA(\lambda)_{slope>1}$ and by (iv) $$\rho(f,H[\eta,M])\to 0\text{ for }\eta\to 0_+,~M\to\infty.$$
For the metric entropy of $H$ from Lemma \ref{l:5} we obtain
\begin{align}h_{\lambda}(H)
& = \int_{0}^1\log\vert H'[\eta,M](x)\vert \, \d\lambda(x) = \sum_{i=0}^{m-1} \int_{0}^{Mq^{-1} \gamma_i} \log | h_i'| \, \d\lambda \nonumber\\ &=  Mq^{-1}
\sum_{i=0}^{m-1} \int_{0}^{\gamma_i} \log | h_i'| \, \d\lambda \nonumber\\ &=M q^{-1} \sum_{i=0}^{m-1} \left (  p_i s \log \frac{1}{1 - \eta} +( \gamma_i - p_i
s) \log \frac{m-1}{\eta}  \right )\nonumber\\
&\label{e:27}=(1-\eta)\log\frac{1}{1-\eta}+\eta\log\frac{m-1}{\eta},
\end{align}
where the last equality follows from (v), Equality \eqref{e:*} and the easily verifiable fact that $Mq^{-1}p_is = \alpha_i(1 - \eta)$.

Thus, for each $M$, for any $c \in (0,\log(m-1))$ the is an $\eta$ such that
the entropy of $H(\eta,M)$  equals $c$.

\begin{figure}[t]
\centering
\includegraphics[width=.5\textwidth]{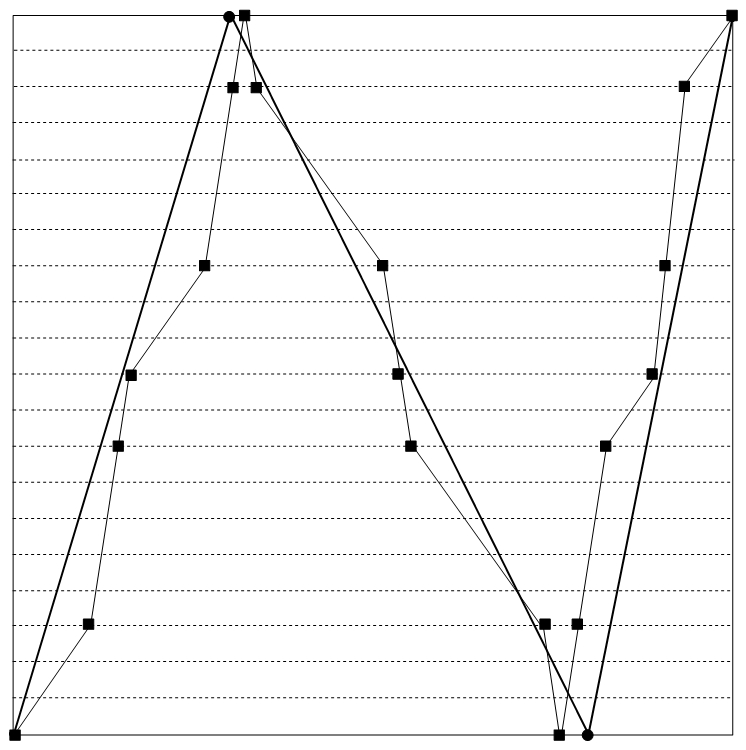}
\caption{$F\sim (+,3/10,1/2,1/5)$, $\eta=3/20$, $q=10$, $M=20$.}\label{fig:dendritemap5}
\end{figure}

\noindent {\bf II.} Fix $f\in PA(\lambda)_{slope>1}$, let $0=y_0<\cdots <y_n=1$ be such that $(y_j,y_{j+1})$, $j=0,\dots,n-1$, are the maximal open intervals on
which the map
$$\card f^{-1}\colon~[0,1]\to \N,~m_j := \card f^{-1}(y) \in \N,~y\in (y_j,y_{j+1})$$
is constant. This map is well defined since $f$ is piecewise affine with (absolute) slopes greater than $1$. Let us denote $\alpha^j_i$,
$i=0,\dots,m_j-1$, the Lebesgue measure
of the $i$th (from the left) connected components of $f^{-1}((y_j,y_{j+1}))$;
Fix the vector of $m_j$'s and $n$ as above, and consider the system of equations
($z_0 =0$, $z_n =1$, all other variables free)
$$\sum_{i=0}^{m_j-1}\beta^j_{i}=z_{j+1}-z_j \hbox{ for each } j \in \{0,\dots,n-1\}.$$
A solution of this equation is a $n -1 + \sum_{j=0}^{n-1}  m_j$-tuple.
Since $f\in C(\lambda)$, $\sum_{i=0}^{m_j-1}\alpha^j_{i}=y_{j+1}-y_j$ for each $j \in \{0,\dots,n-1\}$,
thus it is  a solution for this system.
But this system has a open set of solutions, so it has solutions (arbitrarily close to the fixed solution)  such that all the numbers $\alpha^j_i$ and  the $y_j$
are rational for all $i,j$.
Each such solution corresponds
to a map in $C(\lambda)$, thus wlog we can replace $f$ by a close ``rational'' map.

Now we  essentially repeat step I for $f$  restricted so that its image is $(y_j,y_{j+1})$. More precisely
for each $j$ we can consider the map $F^{(j)}\sim (*,\alpha^j_0,\dots,\alpha^j_{m_j-1})$
with $\sum_{i=0}^{m_j-1}\alpha^j_{i}<1$ , i.e., $$F^{(j)}\colon~[0,\sum_{i=0}^{m_j-1}\alpha^j_{i}]\to [0,\sum_{i=0}^{m_j-1}\alpha^j_{i}],~$$
where $*=+$, resp. $*=-$ if $f$ increases, resp. decreases on the leftmost connected components of $f^{-1}([y_j,y_{j+1}])$. Let $\alpha^j_i=p^j_i/q$. Using part
I  with $\eta\in (0,1)$ and
$M$ divisible by $q$, for each $j$ there is a map $H^{(j)}[\eta,M]$ that approximates (for small $\eta$ and large $M$) the map $F^{(j)}$. Moreover, each
$H^{(j)}$ is composed from
$h^{(j)}_i$, $i=0,\dots,m_j-1$ and we can use those maps to approximate the map $f$ by the uniquely determined map $H = H[\eta,M]\colon~[0,1]\to [0,1]$ in the
following way.

Let  $C^j_i$ denote the $i$th connected component of $f^{-1}([y_j,y_{j+1}])$.
 Remember from step I that the graph of $H^{(j)}$ is produced by gluing various horizontally
shifted  copies of $h^{(j)}_i$.
The graph of $H$ is produced by gluing various copies of the same pieces, but with different vertical  (with respect to $j$)
and horizontal (with respect to $C_i^j)$ shifts.  More precisely
the copies of $H_i^{(j)} := h^{(j)}_i + y_j$ are glued in the same combinatorial order of the interval $C^j_i$, i.e.,
 the rightmost value (either $y_j$ or $y_{j+1}$) of preceding $H^{(j)}_i$ coincides with the leftmost value of the following $H^{(j')}_{i'}$.

Note that  $H[\eta,M]\in PA(\lambda)_{slope>1}$, Equality (\ref{e:27}) holds,
$$h_{\lambda}(H[\eta,M])\to 0 \hbox{ for }\eta\to 0_+$$
and
$$\rho(f,H[\eta,M])\to 0 \hbox{ for } \eta\to 0_+ \hbox{ and } M\to\infty.$$

As in part I, for each $M$, for any $c \in (0,\log(m-1))$ there is an $\eta$ such that
the entropy of $H(\eta,M)$  equals $c$.
In particular,  the proof of (\ref{e:25}) is finished.

\noindent {\bf III.} Fix a map $f\in PA(\lambda)_{slope>1}$ and $\eps$ and $\delta$ positive. Using
part II, for a sufficiently small $\eta$ and large $M$, the map $H=H[\eta,M]\in PA(\lambda)_{slope>1}$ satisfies
$$\rho(f,H)<\delta/2,~\int_{0}^1\log\vert H'(x)\vert~\d\lambda(x)=(1-\eta)\log\frac{1}{1-\eta}+\eta\log\frac{m-1}{\eta}<\eps/2.
$$
If $H$ is not Markov, we can use sufficiently small window perturbations of the piecewise affine map $H$ analogous to the ones from the proofs of Propositions
\ref{p:3} and \ref{p:6} to
obtain $\tilde H\in PAM(\lambda)_{slope>1}$ still satisfying
$$\rho(f,\tilde H)<\delta,~h_{\lambda}({\tilde H}) = \int_{0}^1\log\vert \tilde H'(x)\vert~\d\lambda(x)<\eps.
$$

Let $S = S(\tilde H) $ be the set consisting of all orbits of points of discontinuity of the derivative of $\tilde H$ and also both endpoints $0,1$. Since
$\tilde H$ is Markov, $S$ is finite
and there exists a periodic orbit $P$ and its two consecutive points $p,p'\in P$ such that $[p,p']\cap S=\emptyset$, (i.e., $\tilde H\upharpoonright [p,p']$ is
affine) and, $p$ and $p'$ are
so close that using Lemma \ref{l:2}, for every $m\ge 3$ any $m$-fold piecewise affine perturbation (not necessarily regular) of $\tilde H$ on $[p,p']$ is still
from $B(f;\delta)$. Notice that
each such perturbation $\hat H$ is again from $PAM(\lambda)_{slope>1}$ hence by Lemma \ref{l:5} the entropy is given by the integral formula and
\begin{itemize}
\item $h_{\lambda}(\hat H)\in \left (h_{\lambda}(\tilde H),h_{\lambda}(\tilde H)+(\log m)(p'-p)\right]$,
 \item  Lemma \ref{l:5} implies that  the entropy $h_{\lambda}(\hat H)$ is a continuous function of the slopes of
    the affine pieces of $\hat H \upharpoonright [p,p']$ and that
    each value from $\left (h_{\lambda}(\tilde H),h_{\lambda}(\tilde H)+(\log m)(p'-p)\right]$ is the entropy of some  piecewise affine $m$-fold perturbation
    $\hat H$ of $\tilde H$ on
    $[p,p']$
\end{itemize}

To see that for every $c\in (0,\infty)$ the set $PAM(\lambda)_{entr=c}$ is dense in $C(\lambda)$, we proceed as
follows.   As mentioned in the beginning of the proof of Proposition \ref{p:3} $PA(\lambda)$ is dense in $C(\lambda)$. For each $f \in PA(\lambda)$ the slope of
each affine piece is at least
one, using a window perturbation we can make an arbitrarily small perturbation replacing these affine pieces with ones whose slopes are strictly greater than
one, obtaining  that
$PA(\lambda)_{slope > 1}$
is dense in $C(\lambda)$.

Fix $c$ and an $g \in C(\lambda)$,  choose $f \in PA(\lambda)_{slope > 1}$ arbitrarily close to $g$.
  By Equation \eqref{e:25} we can find a ${H} \in PA(\lambda)_{slope > 1} $ arbitrarily close of $g$ with entropy strictly  less than $c$. The above construction
  yields Markov map $\tilde H$
  with small entropy, and
  for large enough $m$ it yields a Markov map $\hat H$  with entropy exactly $c$.
  \end{proof}

\begin{proposition}
The set $C(\lambda)_{entr=+\infty}$ is dense in $C(\lambda)$.
\end{proposition}

\begin{proof} We proceed like in the proof of the previous lemma. We fix $g \in C(\lambda)$, and we repeat steps I and II,
and then in step III  we realize a sequence of
window perturbations:  sequences $(H_n)_{n\ge 1}$, $([p_n,p_n'])_{n\ge 1}$ and $(m_n)_{n\ge 1}$
such that for each $n$,
\begin{itemize}
\item $[p_n,p_n']\supset [p_{n+1},p'_{n+1}]$,
\item $H_{n+1}$ is a $m_n$-fold window perturbation of $H_n$ on $[p_n,p_n']$,
\item $H_{n}\in PAM(\lambda)_{slope>1}$,
\item $h_{\lambda}(H_n)=h_{\lambda}(H_n,\A_n)>n$, where $\A_n$ is a Markov partition for $H_n$,
\item for some $H_{\infty}\in B(g;\delta)$, $\rho(H_n,H_{\infty})\to 0$ for $n\to\infty$,
\item $h_{\lambda}(H_{\infty},\A_n)\ge h_{\lambda}(H_n,\A_n)>n$ hence $h_{\lambda}(H_{\infty})=\infty$.
\end{itemize}
\end{proof}

For completeness we prove the following  fact, which is well known in many situations.

\begin{proposition}\label{prop:inf}
The set $C(\lambda)_{h_{top} = \infty}$ is a dense $G_{\delta}$ subset of $C(\lambda)$.
\end{proposition}

\begin{proof}
Every map $f \in C(\lambda) \setminus \{id\}$,  has a fixed point $b$ where the graph of $f$ is transverse to the diagonal at $b$.
Using an $(n+2)$-fold window perturbation on a neighborhood of $b$, we can create a map $g \in C(\lambda)$ arbitrarily close to $f$ with a horseshoe with
entropy $\log n$ in the  window.  Since horseshoes are stable under perturbations,  there is an open ball $B(g,\delta)$  such that each $h$ in this ball has
topological entropy at least $\log
n$ for any $n \ge 1$.
\end{proof}

\end{document}